\batchmode
\makeatletter
\def\input@path{{"C:/Users/magno/Dropbox/Universita/Ricerche/Asymptotic Lines/Article/"}}
\makeatother
\documentclass[oneside,english]{amsart}
\usepackage[T1]{fontenc}
\usepackage[latin9]{inputenc}
\usepackage{babel}
\usepackage{units}
\usepackage{textcomp}
\usepackage{amstext}
\usepackage{amsthm}
\usepackage{amssymb}
\usepackage{graphicx}
\usepackage[all]{xy}
\usepackage[unicode=true,
 bookmarks=true,bookmarksnumbered=false,bookmarksopen=false,
 breaklinks=false,pdfborder={0 0 1},backref=page,colorlinks=false]
 {hyperref}
\hypersetup{pdftitle={Osculating Conics},
 pdfauthor={Giosuè Muratore}}

\makeatletter
\numberwithin{equation}{section}
\numberwithin{figure}{section}
\numberwithin{table}{section}
\theoremstyle{plain}
\newtheorem{thm}{\protect\theoremname}[section]
\theoremstyle{definition}
\newtheorem{defn}[thm]{\protect\definitionname}
\theoremstyle{definition}
\newtheorem{example}[thm]{\protect\examplename}
\theoremstyle{plain}
\newtheorem{fact}[thm]{\protect\factname}
\theoremstyle{plain}
\newtheorem{lem}[thm]{\protect\lemmaname}
\theoremstyle{plain}
\newtheorem{prop}[thm]{\protect\propositionname}
\theoremstyle{remark}
\newtheorem{rem}[thm]{\protect\remarkname}


\makeatother

\providecommand{\definitionname}{Definition}
\providecommand{\examplename}{Example}
\providecommand{\factname}{Fact}
\providecommand{\lemmaname}{Lemma}
\providecommand{\propositionname}{Proposition}
\providecommand{\remarkname}{Remark}
\providecommand{\theoremname}{Theorem}

\begin{document}
\global\long\def\ra{\rightarrow}%
\global\long\def\M#1{\overline{M}_{0,#1}}%
\global\long\def\Mbar{\overline{M}}%
\global\long\def\ev{\mathrm{ev}}%
\global\long\def\Z{\mathbb{Z}}%
\global\long\def\Q{\mathbb{Q}}%
\global\long\def\OC{\mathrm{OC}}%

\date{\today}
\title{a recursive formula for osculating curves}
\author{Giosuè Muratore}
\address{Department of Mathematics, Universidade Federal de Minas Gerais, Belo
Horizonte, MG, Brazil.}
\begin{abstract}
Let $X$ be a smooth complex projective variety. Using a construction
devised to Gathmann, we present a recursive formula for some of the
Gromov-Witten invariants of $X$. We prove that, when $X$ is homogeneous,
this formula gives the number of osculating rational curves at a general
point of a general hypersurface of $X$. This generalizes the classical
well known pairs of inflexion (asymptotic) lines for surfaces in $\mathbb{P}^{3}$
of Salmon, as well as Darboux\textquoteright s $27$ osculating conics.
\end{abstract}

\email{muratore.g.e@gmail.com}
\maketitle

\section{Introduction}

One of the main tool of modern enumerative geometry is the theory
of GW invariants and quantum cohomology rings, provided by ideas from
string theory. This tool permits to encode intersection conditions
on curves, as well as tangency conditions, to some calculations of
intersection of cycles. One of the milestones in the theory is Kontsevich's
recursive formula for the number of plane rational curves of degree
$d$ through $3d-1$ general points. Prior to that, no hint had been
suspected that the individual answers for each $d$ were related whatsoever. 

Our enumerative problem is the following. Fix a general smooth hypersurface
$Y$ of degree $d\ge4$ in $\mathbb{P}^{3}$. We want to count the
number of rational curves of degree $n$ in $\mathbb{P}^{3}$ with
contact order $4n-1$ at a general point of $Y$. The number of such
curves is finite. This problem is rooted in classical geometry. Salmon
knew that there are exactly two lines meeting a surface in $\mathbb{P}^{3}$
in three ``coincidence points'' \cite[§265]{salmon}. Darboux \cite[p. 372]{BSMA_1880_2_4_1_348_1}
proved that a general surface $Y$ in $\mathbb{P}^{3}$ has $27$
osculating conics, i.e., conics with contact order at least $7$ at
a general point $y$. In the case $Y$ is a cubic, he pointed out
that a conic with such a multiplicity must be contained in $Y$, so
it reduces to the residual intersection of the plane spanned by $y$
and any of the $27$ lines of $Y$. His argument for counting osculating
conics in the case that $Y$ has degree at least $4$ rests on an
intricate application of classical elimination theory, grant in vogue
then, alas not so clear for myself. 

We have learnt a nice application of inflexional lines in \cite[Theorem 4.2]{MR3877435}.
These lines are used to construct a $2$-web on $Y$. This shows a
bound on the number of lines on $Y$. Darboux's $27$ conics define
a $27$-web, that potentially could give an upper bound to the number
of conics on $Y$. Inspired by these results, we try to study osculating
curves in greater generality.
\begin{defn}
Let $X$ be a smooth complex projective variety, let $\beta$ be the
homological class of a curve and let $Y$ be a very ample smooth hypersurface
$Y\subset X$. An osculating curve $C$ of class $\beta$ is an irreducible
rational curve in $X$, not contained in $Y$, such that the intersection
index at a general point of $Y$ with $C$ is at least $c_{1}(X)\cdot\beta-1$.
We denote by $\OC(\beta,X)$ the number of osculating curves in $X$
of class $\beta$ through a general point of $Y$.
\end{defn}

In this paper we will find a formula to compute the number of osculating
curves for certain $X$ and $\beta$. Using Gathmann's construction,
we will find a recursive formula for a Gromov-Witten invariant of
$X$ relative to $Y$ (Equation (\ref{eq:abx})) under the hypothesis
that $Y$ has no rational curves. That assumption assures that the
result of such formula does not depend on $Y$. Moreover, we will
prove that this invariant coincides with $\OC(\beta,X)$ (i.e., it
is enumerative) when $X$ is homogeneous (Proposition \ref{lem:Enumerative}).
These results imply the following.
\begin{thm}
\label{prop:Centralresult}Let $X$ be a homogeneous variety, let
$\beta$ be the homological class of a curve. There exists a recursive
formula for the number of curves of class $\beta$ osculating a very
ample hypersurface.
\end{thm}

For related results, see for example \cite{FanHu} and reference therein.

The paper is organized as follows. Section \ref{sec:Kontsevich-moduli-space}
and \ref{sec:Gathmann-construction} recall standard notations of
the moduli space of stable curves, Gromov-Witten invariants and Gathmann's
construction of the moduli space of curves with tangency conditions.
In Section \ref{sec:Osculating-curves}, we will study the connection
between osculating curves and Gromov-Witten invariants. Section \ref{sec:Osculating-conics}
contains the proof of the recursive formula cited before. Finally
Section \ref{sec:Applications} contains some applications. In particular,
we present an implementation of $\OC(\beta,X)$ in case $X$ is a
product of projective spaces.

The author would like to thank Angelo Lopez and Eduardo Esteves for
numerous fruitful discussions, and Jorge Vitório Pereira for pointing
me out the connection with webs. I thank Andreas Gathmann for his
help. I especially thank Israel Vainsencher for calling my attention
on this problem and for his constant support during this year at UFMG.
The author is supported by postdoctoral fellowship PNPD-CAPES.

\section{Kontsevich Moduli Space of Stable Maps}

\label{sec:Kontsevich-moduli-space}We begin by giving an informal
discussion of the main properties of the Kontsevich moduli space of
stable maps, following \cite{MR1492534} and \cite{hori2003mirror}.
Let $X$ be a smooth complex projective variety, let $\beta\in H_{2}(X,\Z)$
be a non torsion homology class, and let $Y\subset X$ be a smooth
very ample hypersurface. We denote by $Y$ or $[Y]$ the cohomology
class of the subvariety $Y$ in $H^{2}(X,\Z)$, given by Poincaré
duality. The cohomology class of a point is denoted by $\mathrm{pt}$.

For any non-negative integer $n$, we denote by $\overline{M}_{0,n}$
and $\overline{M}_{0,n}(X,\beta)$ the moduli spaces of $n$-pointed
genus zero stable curves and stable maps to $X$ of class $\beta$,
respectively. The markings provide evaluation morphisms $\ev_{i}:\overline{M}_{0,n}(X,\beta)\ra X$.
We have tautological classes $\psi_{i}:=c_{1}(\mathbb{L}_{i})$ where
$\mathbb{L}_{i}$ is the line bundle whose fiber at a stable map $(C,p_{1},...,p_{n},f)$
is the cotangent line to $C$ at point $p_{i}$. When $n=1$, we omit
the index.
\begin{defn}
The virtual dimension of $\overline{M}_{0,n}(X,\beta)$ is the number
\[
\mathrm{vdim}\overline{M}_{0,n}(X,\beta)=\dim X+c_{1}(X)\cdot\beta+n-3.
\]
\end{defn}

Since $X$ is projective, there exists a homology class, the virtual
fundamental class $[\overline{M}_{0,n}(X,\beta)]^{virt}$, of dimension
$\mathrm{vdim}\overline{M}_{0,n}(X,\beta)$ (see \cite[Chapter 26]{hori2003mirror}
for further discussions). If $X$ is also a homogeneous variety (i.e.,
a quotient $\nicefrac{G}{P}$, where $G$ is a Lie group and $P$
is a parabolic subgroup), and $\overline{M}_{0,n}(X,\beta)\neq\emptyset$,
then $\overline{M}_{0,n}(X,\beta)$ exists as a projective non singular
stack or orbifold coarse moduli space of pure dimension $\mathrm{vdim}\overline{M}_{0,n}(X,\beta)$
\cite[Theorem 1,2,3]{MR1492534}.
\begin{defn}
\label{def:descendents-invariants}For every choice of classes $\gamma_{1},...,\gamma_{n}\in H^{*}(X,\Z)$,
and non negative integers $a_{1},...,a_{n}\in\Z$ such that $\sum_{i=1}^{n}\mathrm{codim}\gamma_{i}+a_{i}=\mathrm{vdim}\overline{M}_{0,n}(X,\beta)$,
we have the numbers
\[
I_{n,\beta}^{X}(\gamma_{1}\psi_{1}^{a_{1}}\otimes\cdots\otimes\gamma_{n}\psi_{n}^{a_{n}}):=\ev_{1}^{*}(\gamma_{1})\cdot\psi_{1}^{a_{1}}\cdot...\cdot\ev_{n}^{*}(\gamma_{n})\cdot\psi_{n}^{a_{n}}\cdot[\overline{M}_{0,n}(X,\beta)]^{virt},
\]
called descendant invariants.
\end{defn}

We can extend this definition to every integer $a_{1},...,a_{n}\in\Z$,
by imposing 
\[
I_{n,\beta}^{X}(\gamma_{1}\psi_{1}^{a_{1}}\otimes\cdots\otimes\gamma_{n}\psi_{n}^{a_{n}})=0
\]
if $a_{i}<0$ for some $i$.

We adopt the well established notation that encodes all $1$-point
invariants of class $\beta$ in a single cohomology class:
\begin{eqnarray*}
I_{1,\beta}^{X} & := & \ev_{*}\left(\frac{1}{1-\psi}\left[\overline{M}_{0,1}(X,\beta)\right]^{virt}\right)\\
 & := & \sum_{i,j}I_{1,\beta}^{X}(T^{i}\psi^{j})\cdot T_{i},
\end{eqnarray*}
where $\{T^{i}\}$ and $\{T_{i}\}$ are bases of $H^{*}(X,\Z)\otimes\mathbb{Q}$
dual to each other. Note that
\[
I_{1,\beta}^{X}(T^{i}\psi^{j})
\]
is zero when $j\neq\mathrm{vdim}\overline{M}_{0,1}(X,\beta)-i$. We
define $I_{1,0}^{X}:=1_{X}$, i.e., the unity of the ring $H^{*}(X,\Z)\otimes\mathbb{Q}$.
\begin{example}
\label{exa:Pand}A very useful descendant invariant is the following.
Let $X=\mathbb{P}^{s}$, then every class $\beta$ will be of the
form $\beta=n[\mathrm{line}]$ for some positive integer $n$. It
is known by \cite[Section 1.4]{SB_1997-1998__40__307_0} that 
\[
I_{1,\beta}^{\mathbb{P}^{s}}(\mathrm{pt})=\frac{1}{(n!)^{s+1}}.
\]
\end{example}

\section{Gathmann Construction}

\label{sec:Gathmann-construction}We recall briefly the construction
given in \cite{gathmann2002absolute,gathmann2003relative}. Let $m$
be a non negative integer. There exists a closed subspace $\overline{M}_{(m)}^{Y}(X,\beta)\subseteq\M 1(X,\beta)$
which parameterizes curves with multiplicity with $Y$ at least $m$
at the marked point. As a set, it has the following simple description.
\begin{defn}[{\cite[Definition 1.1]{gathmann2002absolute}}]
\label{def:MY}The space $\overline{M}_{(m)}^{Y}(X,\beta)$ is the
locus in $\M 1(X,\beta)$ of all stable maps $(C,p,f)$ such that
\begin{enumerate}
\item $f(p)\in Y$ if $m>0$.
\item $f^{*}Y-mp$ in the Chow group $A_{0}(f^{-1}(Y))$ is effective.
\end{enumerate}
\end{defn}

Curves with multiplicity $0$ are just unrestricted curves in $X$,
whereas a multiplicity of $Y\cdot\beta+1$ forces at least the irreducible
curves to lie inside $Y$. This space comes equipped with a virtual
fundamental class $[\overline{M}_{(m)}^{Y}(X,\beta)]^{virt}$ of dimension
$\mathrm{vdim}(\M n(X,\beta))-m$. 

The explicit form of $[\overline{M}_{(m)}^{Y}(X,\beta)]^{virt}$ is
given by the following
\begin{thm}[{\cite[Theorem 0.1]{gathmann2003relative}}]
\label{thm:Gathmann}For all $m\ge0$ we have 
\begin{equation}
(m\psi+\ev^{*}Y)\cdot\left[\overline{M}_{(m)}^{Y}(X,\beta)\right]^{virt}=\left[\overline{M}_{(m+1)}^{Y}(X,\beta)\right]^{virt}+\left[D_{(m)}^{Y}(X,\beta)\right]^{virt}.\label{eq:Gathmann_little}
\end{equation}
Here, the correction term $D_{(m)}^{Y}(X,\beta)=\coprod_{r}\coprod_{B,M}D^{Y}(X,B,M)$
is a disjoint union of individual terms
\[
D^{Y}(X,B,M):=\M{1+r}(Y,\beta^{(0)})\times_{Y^{r}}\prod_{i=1}^{r}\Mbar_{(m^{(i)})}^{Y}(X,\beta^{(i)})
\]
 where $r\ge0$, $B=(\beta^{(0)},...,\beta^{(r)})$ with $\beta^{(i)}\in H_{2}(X)$/torsion
and $\beta^{(i)}\neq0$ for $i>0$, and $M=(m^{(1)},...,m^{(r)})$
with $m^{(i)}>0$. The maps to $Y^{r}$ are the evaluation maps for
the last $r$ marked points of $\Mbar_{1+r}(Y,\beta^{(0)})$ and each
of the marked points of $\Mbar_{(m^{(i)})}^{Y}(X,\beta^{(i)})$, respectively.
The union in $D_{(m)}^{Y}(X,\beta)$ is taken over all $r$, $B$,
and $M$ subject to the following three conditions:
\begin{eqnarray*}
\sum_{i=0}^{r}\beta^{(i)}=\beta &  & \mathrm{(degree\,condition)}\\
Y\cdot\beta^{(0)}+\sum_{i=1}^{r}m^{(i)}=m &  & \mathrm{(multiplicity\,condition)}\\
\mathrm{if}\,\beta^{(0)}=0\,\mathrm{then}\,r\ge2. &  & \mathrm{(stability\,condition)}
\end{eqnarray*}

In (\ref{eq:Gathmann_little}), the virtual fundamental class of the
summands $D^{Y}(X,B,M)$ is defined to be $\frac{m^{(1)}\cdots m^{(r)}}{r!}$
times the class induced by the virtual fundamental classes of the
factors $\M{1+r}(Y,\beta^{(0)})$ and $\Mbar_{(m^{(i)})}^{Y}(X,\beta^{(i)})$.
The spaces $D^{Y}(X,B,M)$ can be considered to be subspaces of $\M 1(X,\beta)$,
so the equation of the theorem makes sense in the Chow group of $\M 1(X,\beta)$.
\end{thm}

Note that this theorem implies immediately the following
\begin{fact}
\label{fact:No_excess}If $D_{(m)}^{Y}(X,\beta)=0$ for all $0\le m\le n$,
then
\[
\left[\overline{M}_{(n+1)}^{Y}(X,\beta)\right]^{virt}=c_{n+1}(\mathcal{P}^{n}(Y)),
\]
where $\mathcal{P}^{n}(Y)$ is the bundle of $n$-jets of $\ev^{*}(\mathcal{O}_{X}(Y))$
with respect to the tautological line bundle $\mathbb{L}$. This follows
from the initial condition 
\[
\left[\overline{M}_{(1)}^{Y}(X,\beta)\right]^{virt}=(0\psi+\ev^{*}Y)\text{·}\left[\overline{M}_{(0)}^{Y}(X,\beta)\right]^{virt}=\ev^{*}Y=c_{1}(\mathcal{P}^{0}(Y)),
\]
and from the exact sequence
\begin{equation}
0\ra\mathbb{L}^{\otimes m}\otimes\ev^{*}(\mathcal{O}_{X}(Y))\ra\mathcal{P}^{m}(Y)\ra\mathcal{P}^{m-1}(Y)\ra0.\label{eq:jetbundle}
\end{equation}
\end{fact}

Finally, we define descendant invariants relative to the spaces $\overline{M}_{(m)}^{Y}(X,\beta)$.
In the following definition, by $\M n(Y,\beta)$ we mean the space
of $n$-pointed stable maps to $Y$ of all homology classes whose
push-forward to $X$ is $\beta$. 
\begin{defn}[{\cite[Section 1]{gathmann2003relative}}]
\label{def:I(gammapsi)}For every $m\ge0$ and $\gamma\in H^{*}(X,\Z)$
we define 
\[
I_{\beta,(m)}(\gamma\psi^{j}):=\ev^{*}(\gamma)\cdot\psi^{j}\cdot[\overline{M}_{(m)}^{Y}(X,\beta)]^{virt},
\]
where $j=\mathrm{vdim}\M 1(X,\beta)-m-\mathrm{codim}\gamma$. We assemble
all those invariants in a unique cohomology class of $X$,
\begin{eqnarray*}
I_{\beta,(m)} & := & \ev_{*}\left(\frac{1}{1-\psi}\left[\overline{M}_{(m)}^{Y}(X,\beta)\right]^{virt}\right)\\
 & := & \sum_{i,j}I_{\beta,(m)}(T^{i}\psi^{j})\cdot T_{i}.\\
J_{\beta,(m)} & := & \ev_{*}\left(\frac{1}{1-\psi}\left[D_{(m)}^{Y}(X,\beta)\right]^{virt}\right)+m\cdot\ev_{*}\left[\overline{M}_{(m)}^{Y}(X,\beta)\right]^{virt}
\end{eqnarray*}
\end{defn}

By $\M n(Y,\beta)$ we mean the space of $n$-pointed stable maps
to $Y$ of all homology classes whose push-forward to $X$ is $\beta$.
For every integer $i=1,...,n$, we denote by $\widetilde{\ev}_{i}:\M n(Y,\beta)\ra Y$
the evaluation maps to $Y$ instead of $X$.
\begin{defn}[{\cite[Definition 5.1]{gathmann2002absolute}}]
\label{def:Relative}For cohomology classes $\gamma_{i}\in H^{*}(X,\Z)$
we define $I_{n,\beta}^{Y}(\gamma_{1}\psi_{1}^{a_{1}}\otimes\cdots\otimes\gamma_{n}\psi_{n}^{a_{n}})$
in the same way of Definition \ref{def:descendents-invariants}, replacing
$\overline{M}_{0,n}(X,\beta)$ by $\overline{M}_{0,n}(Y,\beta)$,
but keeping the $\ev_{i}$ to denote the evaluation maps to $X$.
More generally, we can take some of the cohomology classes $\gamma_{i}$
to be classes of $Y$ instead of $X$. In that case we simply use
$\widetilde{\ev}_{i}(\gamma_{i})$ instead of $\ev_{i}(\gamma_{i})$. 
\end{defn}

By construction $I_{\beta,(0)}=I_{1,\beta}^{X}$. From (\ref{eq:Gathmann_little})
it follows
\begin{lem}[{\cite[Lemma 1.2]{gathmann2003relative}}]
For all torsion free effective class $\beta\neq0,$ and $m\ge0$
we have

\begin{equation}
(Y+m)\cdot I_{\beta,(m)}=I_{\beta,(m+1)}+J_{\beta,(m)}\,\,\,\,\in H^{*}(X,\Z).\label{eq:IandJ}
\end{equation}
\end{lem}

The number $m$ in $Y+m$ is to be taken as $m\cdot1_{X}$.

A construction very similar to Gathmann's was used \textit{ante litteram}
by Kock for counting bitangents of a plane curve. See \cite{Kock,MR2313122}
for more details.

\section{Osculating Curves}

\label{sec:Osculating-curves}We denote by $C_{\beta}$ the constant
\[
C_{\beta}:=c_{1}(X)\cdot\beta-2.
\]
This constant has the property that the virtual fundamental class
of $\overline{M}_{(C_{\beta}+1)}^{Y}(X,\beta)$ has the same dimension
of $Y$.

When $X$ is homogeneous, there is a smooth dense open subspace $M_{0,1}(X,\beta)$
in $\overline{M}_{0,1}(X,\beta)$ whose points are irreducible stable
maps with no non-trivial automorphisms \cite[Lemma 13\&Theorem 2]{MR1492534}.
We denote by $M_{0,1}(X,\beta)^{*}$ the (possibly empty) open subspace
of $M_{0,1}(X,\beta)$ whose points are birational maps. The following
proposition clarifies the enumerative meaning of $[\overline{M}_{(C_{\beta}+1)}^{Y}(X,\beta)]^{virt}$.
\begin{prop}
\label{lem:Enumerative}Let $X$ be a homogeneous variety, and let
$Y\subset X$ be a general hypersurface which does not contain rational
curves. Let $T$ be any $\Q$-cohomological class, of codimension
$\dim X-1$, such that $T\cdot Y=\mathrm{pt}$. Then the number of
osculating curves at a general point of $Y$ is $I_{\beta,(C_{\beta}+1)}(T)$.
\end{prop}

\begin{proof}
Since we have just one marked point, there are no different labeling
of the marked points that give the same osculating curve. 

Let $s\in\Gamma(X,\mathcal{O}_{X}(Y))$ be the global section defining
$Y$. It defines a global section $\partial(s)$ of the jet bundle
$\mathcal{P}^{C_{\beta}}(Y)$. We know that $\overline{M}_{0,1}(X,\beta)$
is irreducible of the expected dimension (Section \ref{sec:Kontsevich-moduli-space}).
The osculating curves are parameterized by those stable maps in $M_{0,1}(X,\beta)^{*}$
at which the section $\partial(s)$ vanishes. The rank of $\mathcal{P}^{C_{\beta}}(Y)$
is $C_{\beta}+1$. By generality of $Y$, and by the hypothesis that
$Y$ has no rational curves, the locus of osculating curves in $M_{0,1}(X,\beta)^{*}$
has codimension $C_{\beta}+1$, which means that it has dimension
$\dim Y$. This locus is contained in $\overline{M}_{(C_{\beta}+1)}^{Y}(X,\beta)$
by Definition \ref{def:MY}. Let $i:Y\ra X$ be the inclusion. By
construction of $\overline{M}_{(C_{\beta}+1)}^{Y}(X,\beta)$, there
is a map $\widetilde{\ev}:\overline{M}_{(C_{\beta}+1)}^{Y}(X,\beta)\ra Y$
which makes the following diagram commutative. 
\[
\xymatrix{\overline{M}_{(C_{\beta}+1)}^{Y}(X,\beta)\ar[r]^{\,\,\,\,\,\,\,\,\,\,\,\,\,\,\,\,\widetilde{\ev}}\ar[dr]^{\ev} & Y\ar[d]^{i}\\
 & X
}
\]
That is, $\widetilde{\ev}$ sends each curve $(C,p,f)$ to the point
of tangency $f(p)\in Y$. If $M_{0,1}(X,\beta)^{*}\neq\emptyset$,
then the moduli space $\overline{M}_{(C_{\beta}+1)}^{Y}(X,\beta)$
has a component $\overline{M}^{*}$ of the expected dimension $\dim Y$,
where each general point represents a stable map in $M_{0,1}(X,\beta)^{*}$
whose image is an osculating curve. We may have another component,
$\overline{M}^{c}$, whose points parameterizes maps $(\mathbb{P}^{1},p,f)$
where $f$ is a not generically injective map. In this case, we can
find a decomposition $f:\mathbb{P}^{1}\overset{g}{\ra}\mathbb{P}^{1}\overset{h}{\ra}X$
with $g:\mathbb{P}^{1}\ra\mathbb{P}^{1}$ a finite cover, and $h:\mathbb{P}^{1}\ra X$
generically injective ($h$ is the normalization of the curve $f(\mathbb{P}^{1})$).

We want to prove that $\widetilde{\ev}(\overline{M}^{c})$ has dimension
strictly less that $\dim Y$. Let us fix a map $(\mathbb{P}^{1},p,f)\in\overline{M}^{c}$
where $g$ is a cover of degree $k\ge2$, and let $m$ be the multiplicity
of intersection of $h$ with $Y$ at $g(p)$. If we denote by $\beta'$
the class $h_{*}[\mathbb{P}^{1}]$, we clearly have $k\beta'=\beta$.
The contact order of $f$ and $Y$ at $p$ is at most $km$, depending
on the degree of ramification of $p$. If we want that $km$ be at
least $C_{\beta}+1=c_{1}(X)\cdot\beta-1$, then clearly $m\ge c_{1}(X)\cdot\beta'=C_{\beta'}+2$.
This implies that $(\mathbb{P}^{1},g(p),h)$ is in $M_{0,1}(X,\beta')^{*}$
and kills a general section of $\mathcal{P}^{C_{\beta'}+1}(Y)$. The
dimension of the zero set of that general section is
\[
\dim X-2+c_{1}(X)\cdot\beta'-(C_{\beta'}+2)=\dim X-2.
\]
This dimension is strictly less than $\dim Y$. Hence there is no
rational irreducible curve of class $\beta'$ through a general point
of $Y$ with multiplicity $C_{\beta'}+2$ at that point. Since $\widetilde{\ev}(\mathbb{P}^{1},p,f)=h(g(p))$,
we deduce that $\widetilde{\ev}(\overline{M}^{c})$ has dimension
strictly less than $\dim Y$, as claimed. So, for a general point
$y\in Y$, the inverse image $\widetilde{\ev}^{-1}(y)$ is supported
on $\overline{M}^{*}$, and it is possibly empty.

This implies by projection formula that every cycle $\tau\in H_{\dim Y}(\overline{M}^{c},\Z)$
is contracted by $\widetilde{\ev}$. So the contribution to $[\overline{M}_{(C_{\beta}+1)}^{Y}(X,\beta)]^{virt}\in H_{\dim Y}(\overline{M}_{(C_{\beta}+1)}^{Y}(X,\beta),\Z)$
from $\overline{M}^{c}$ does not intersect $\widetilde{\ev}^{*}(\mathrm{pt})$.
Therefore
\begin{equation}
\widetilde{\ev}^{*}(\mathrm{pt})\cdot\left[\overline{M}_{(C_{\beta}+1)}^{Y}(X,\beta)\right]^{virt}=\widetilde{\ev}^{*}(\mathrm{pt})\cdot\left[\overline{M}^{*}\cup\overline{M}^{c}\right]^{virt}=\widetilde{\ev}^{*}(\mathrm{pt})\cdot\left[\overline{M}^{*}\right]^{virt}.\label{eq:ev(T)Mstar2}
\end{equation}
But $\overline{M}^{*}$ has the expected dimension, so its virtual
fundamental class coincides with the usual fundamental class. Let
$T$ be the $\Q$-cohomology class of the statement. Using (\ref{eq:ev(T)Mstar2})
we get
\begin{eqnarray*}
\ev^{*}(T)\cdot\left[\overline{M}_{(C_{\beta}+1)}^{Y}(X,\beta)\right]^{virt} & = & \widetilde{\ev}^{*}(i^{*}(T))\cdot\left[\overline{M}_{(C_{\beta}+1)}^{Y}(X,\beta)\right]^{virt}\\
 & = & \widetilde{\ev}^{*}(\mathrm{pt})\cdot\left[\overline{M}_{(C_{\beta}+1)}^{Y}(X,\beta)\right]^{virt}\\
 & = & \widetilde{\ev}^{*}(\mathrm{pt})\cdot\left[\overline{M}^{*}\right].
\end{eqnarray*}
We deduce that if $\overline{M}^{*}=\emptyset$, then $I_{\beta,(C_{\beta}+1)}(T)=0$.
If $\overline{M}^{*}\neq\emptyset$, then $I_{\beta,(C_{\beta}+1)}(T)$
is equal to the degree of the map $\widetilde{\ev}_{|\overline{M}^{*}}:\overline{M}^{*}\ra Y$,
i.e., to the number of osculating curves through a general point of
$Y$. 
\end{proof}
\begin{example}
\label{exa:osculatingOfPlaneCurve}Take $X=\mathbb{P}^{2}$, $Y$
a general curve of degree $d>2$ and $\beta$ the class of a conic,
so that $C_{\beta}+1=5$. It is clear that we have just one osculating
conic at every point of $Y$. Because if we had two, then every curve
in the linear system that they span would be an osculating conic.
Let $l$ be the tangent line at a general point $y\in Y$. A double
cover of $l$ branched at $y$ will have multiplicity $4$, hence
it is not osculating. If we take $y$ to be a flex point, then a double
cover of $l$ will have multiplicity $6$, so it is osculating at
$y$. But the flex points are not dense in $Y$. This implies that
$\overline{M}_{(C_{\beta}+1)}^{Y}(X,\beta)$ has the following components:
$\overline{M}^{*}$ which is mapped isomorphically to $Y$ by $\widetilde{\ev}$,
and a $1$-dimensional irreducible component for each flex point $y$.
Such a component parameterizes double covers of the tangent $l$ at
$y$, branched at $y$.
\end{example}

Let $d$ be a positive integer, and $Y'$ a general element in the
linear system $|\mathcal{O}_{X}(dY)|$. For every non zero effective
$1$-cycle $\gamma$ of $Y'$, by adjunction 
\[
-K_{Y'}\cdot\gamma=(-K_{X}-dY)_{|Y'}\cdot\gamma
\]
will be negative for some large $d$. Indeed, as $Y$ is very ample,
$Y_{|Y'}\cdot\gamma>0$ by Kleiman's Positivity Theorem \cite[Chapter 3, §1]{kleiman1966toward}.
This implies that 
\[
\mathrm{vdim}\Mbar_{0,1}(Y',\gamma)<0
\]
for $d\gg0$, i.e., $\Mbar_{0,1}(Y',\gamma)$ is virtually empty.
It can happen that the requested intersection multiplicity between
$Y$ and the osculating curve is so high that the curve must be contained
in $Y$ (take lines tangent to linear subspaces). In order to avoid
that, we could substitute $Y$ with $Y'$ for $d\gg0$. We will see
that $I_{\beta,(C_{\beta}+1)}(T)$ does not depend on $d$, as well
as on $Y$, as long as $Y$ has no rational curves. So, it makes sense
to omit $Y$ in $\OC(\beta,X)$.

\section{Recursive Formula}

\label{sec:Osculating-conics}In this section we will give a recursive
formula for $I_{\beta,(C_{\beta}+1)}(T)$. We use the same notation
as before. The variety $Y\subset X$ is a smooth very ample hypersurface
with no rational curves. We suppose that $1_{X},Y\in\{T^{i}\}$. We
denote by $T$ the dual of $Y$ in $\{T_{i}\}$. When $Y$ generates
$H^{2}(X,\Z)\otimes\Q$, the class $T$ is uniquely determined. The
number $I_{\beta,(C_{\beta}+1)}(T)$ can be described as the coefficient
of $Y$ in $I_{\beta,(C_{\beta}+1)}$, as seen in Definition \ref{def:I(gammapsi)}.

We apply Theorem \ref{thm:Gathmann} to compute $I_{\beta,(C_{\beta}+1)}(T)$.
Using Equation \ref{eq:IandJ} we have
\begin{eqnarray}
I_{\beta,(C_{\beta}+1)} & = & (Y+C_{\beta})I_{\beta,(C_{\beta})}-J_{\beta,(C_{\beta})}\nonumber \\
 & = & (Y+C_{\beta})(Y+C_{\beta}-1)I_{\beta,(C_{\beta}-1)}-(Y+C_{\beta})J_{\beta,(C_{\beta}-1)}-J_{\beta,(C_{\beta})}\nonumber \\
 & \vdots & \vdots\nonumber \\
 & = & \left(\prod_{i=0}^{C_{\beta}}(Y+i)\right)I_{\beta,(0)}-\sum_{i=0}^{C_{\beta}-1}\left(\prod_{j=i+1}^{C_{\beta}}(Y+j)\right)J_{\beta,(i)}-J_{\beta,(C_{\beta})}.\label{eq:Bigsum}
\end{eqnarray}
Let us compute the first term of this sum in Equation (\ref{eq:FirstTerm}).
In the next display, following \cite{gathmann2003relative}, $mod\,H^{3}$
means that we omit cohomology classes of codimension greater than
$2$. 
\begin{eqnarray}
\left(\prod_{i=0}^{C_{\beta}}(Y+i)\right)I_{\beta,(0)} & = & \left(\prod_{i=1}^{C_{\beta}}(Y+i)\right)YI_{\beta,(0)}\nonumber \\
 & = & \left(\prod_{i=1}^{C_{\beta}}(Y+i)\right)YI_{1,\beta}^{X}(\mathrm{pt})\,\,\,\,\,(mod\,H^{3})\nonumber \\
 & = & C_{\beta}!I_{1,\beta}^{X}(\mathrm{pt})\cdot Y\,\,\,\,\,(mod\,H^{3}).\label{eq:FirstTerm}
\end{eqnarray}
Back in (\ref{eq:Bigsum}), we need to find the contribution of each
$J_{\beta,(i)}$. Gathmann computed explicitly all the $J_{\beta,(i)}$
in \cite[Lemma 1.8]{gathmann2003relative} using that $-K_{Y}$ is
nef. Since we want that $Y$ has no rational curves, we need that
$-K_{Y}\cdot\gamma$ is negative for every rational curve $\gamma$
in $Y$, as explained at the end of Section \ref{sec:Osculating-curves}.
So $-K_{Y}$ is not nef.
\begin{rem}
\label{rem:H0andH2part}For each $0\le m\le C_{\beta}-1$, dimensional
reasons ensures that the class $\ev_{*}[\Mbar_{(m)}(X,\beta)]^{virt}$
has trivial $H^{0}$ and $H^{2}$ part. Moreover $\ev_{*}[\Mbar_{(C_{\beta})}(X,\beta)]^{virt}$
has trivial $H^{2}$ part.
\end{rem}

Let us compute the contribution of $\ev_{*}\left(\frac{1}{1-\psi}[D_{(m)}^{Y}(X,\beta)]^{virt}\right)$.

Let $D:=D^{Y}(X,B,M)$ be one of the individual term as in Theorem
\ref{thm:Gathmann}, with $B=(\beta^{(0)},...,\beta^{(r)})$ and $M=(m^{(1)},...,m^{(r)})$,
$m^{(i)}>0$. If $\beta^{(0)}\neq0$, then $\Mbar_{0,1}(Y,\beta^{(0)})$
is empty by our hypothesis that $Y$ has no rational curves. So that
$D$ has no contribution. Let $\beta^{(0)}=0$, in particular we have
the following conditions on $D$:
\begin{eqnarray*}
\sum_{i=1}^{r}\beta^{(i)}=\beta &  & \mathrm{(degree\,condition)}\\
\sum_{i=1}^{r}m^{(i)}=m &  & \mathrm{(multiplicity\,condition)}\\
r\ge2. &  & \mathrm{(stability\,condition)}.
\end{eqnarray*}
The value of $\ev_{*}\left(\frac{1}{1-\psi}[D]^{virt}\right)$ is
given by the formula \cite[Remark 1.4,Eq.(2)]{gathmann2003relative}
\begin{equation}
\sum I_{0}^{Y}(T^{i}\psi^{j}\otimes\gamma_{1}\otimes\cdots\otimes\gamma_{r})\cdot\frac{1}{r!}\prod_{k=1}^{r}\left(m^{(k)}\cdot I_{\beta^{(k)},(m^{(k)})}(\gamma_{k}^{\vee})\right)\cdot T_{i},\label{eq:Contribuition}
\end{equation}
where the $\gamma_{k}$ run in a basis of the part of $H^{*}(Y)\otimes\Q$
induced by $X$ \cite[Remark 5.4]{gathmann2002absolute}, and $\gamma_{k}^{\vee}$
is the dual as a $\Q$-class in $X$. By Lefschetz Hyperplane Theorem,
we can take such a basis as $\{T_{|Y}^{i}\}$. We will look for the
conditions on $D$ such that this contribution is non zero.
\begin{lem}
\label{lem:The-multiple-of}The coefficient of $Y$ in $\ev_{*}\left(\frac{1}{1-\psi}[D]^{virt}\right)$
is
\[
\frac{1}{r!}\prod_{k=1}^{r}\left((C_{\beta^{(k)}}+1)I_{\beta^{(k)},(C_{\beta^{(k)}}+1)}(T)\right)
\]
if $m=C_{\beta}+2-r$, and zero otherwise.
\end{lem}

\begin{proof}
The coefficient of $Y$ is given by the sum in Equation (\ref{eq:Contribuition}),
when $T^{i}=T$, $\gamma_{k}\in\{T_{|Y}^{i}\}$ for $1\le k\le r$,
and
\begin{eqnarray*}
j & = & \dim[D]^{virt}-\mathrm{codim}(T)\\
 & = & \mathrm{vdim}\Mbar_{(m+1)}(X,\beta)-(\dim X-1)\\
 & = & \dim X+c_{1}(X)\cdot\beta-2-(m+1)-\dim X+1\\
 & = & C_{\beta}-m.
\end{eqnarray*}
By Definition \ref{def:Relative}, we know that
\begin{equation}
I_{0}^{Y}(T\psi^{j}\otimes\gamma_{1}\otimes\cdots\otimes\gamma_{r})=\ev_{1}^{*}(T)\cdot\psi^{j}\cdot\widetilde{\ev}_{2}^{*}(\gamma_{1})\cdot...\cdot\widetilde{\ev}_{r+1}^{*}(\gamma_{r})\cdot[\overline{M}_{0,1+r}(Y,0)]^{virt}.\label{eq:IY0}
\end{equation}
It is well known that $\overline{M}_{0,1+r}(Y,0)\cong\overline{M}_{0,1+r}\times Y$,
and each map $\widetilde{\ev}_{i}$ is the second projection. Moreover,
$\ev_{1}^{*}(T)=\widetilde{\ev}_{i}(\mathrm{pt})$. So for every $i$,
\[
\ev_{1}^{*}(T)\cdot\widetilde{\ev}_{i}^{*}(1_{X|Y})=\widetilde{\ev}_{i}(\mathrm{pt}\cdot1_{Y})=\widetilde{\ev}_{i}(\mathrm{pt}).
\]
If one of the $\gamma_{i}$ is not $1_{X|Y}=1_{Y}$, then $\ev_{1}^{*}(T)\cdot\widetilde{\ev}_{i+1}^{*}(\gamma_{i})=0$
for dimensional reasons. This implies that the expression (\ref{eq:IY0})
is zero if $\gamma_{k}\neq1_{Y}$ for some $1\le k\le r$.

Now we want to compute $I_{0}^{Y}(T\psi^{j}\otimes1_{Y}\otimes\cdots\otimes1_{Y})$,
where $1_{Y}$ appears $r$ times. Using $(r-2)$-times the string
equation \cite[1.2.I]{SB_1997-1998__40__307_0}, we get
\begin{eqnarray*}
I_{0}^{Y}(T\psi^{j}\otimes1_{Y}^{\otimes r}) & = & I_{0}^{Y}(T\psi^{j-1}\otimes1_{Y}^{\otimes r-1})\\
 & \vdots & \vdots\\
 & = & I_{0}^{Y}(T\psi^{j-(r-2)}\otimes1_{Y}\otimes1_{Y}).
\end{eqnarray*}
Hence,
\[
I_{0}^{Y}(T\psi^{j-(r-2)}\otimes1_{Y}\otimes1_{Y})=T_{|Y}\cdot1_{Y}\cdot1_{Y}\cdot\psi^{j-(r-2)}\cdot\left[\overline{M}_{0,3}\times Y\right].
\]
This expression is $1$ if $j-(r-2)=0$, and $0$ otherwise. We proved
that the coefficient of $Y$ is zero if $C_{\beta}-m-(r-2)\neq0$.
So, necessarily $m=C_{\beta}+2-r$, as asserted.

We need to determine the value of $m^{(k)}$ in $m^{(k)}I_{\beta^{(k)},(m^{(k)})}(\gamma_{k}^{\vee})$.
First of all, since $\gamma_{k}=1_{Y}$, then $\gamma_{k}^{\vee}=T$,
hence 
\[
m^{(k)}I_{\beta^{(k)},(m^{(k)})}(\gamma_{k}^{\vee})=m^{(k)}I_{\beta^{(k)},(m^{(k)})}(T).
\]
Consider
\begin{equation}
I_{\beta^{(k)},(m^{(k)})}(T)=\ev^{*}(T)\cdot\psi^{s}\cdot[\Mbar_{(m^{(k)})}(X,\beta^{(k)})]^{virt},\label{eq:I}
\end{equation}
where
\begin{eqnarray*}
s & = & \dim[\Mbar_{(m^{(k)})}(X,\beta^{(k)})]^{virt}-\mathrm{codim}(T)\\
 & = & C_{\beta^{(k)}}-m^{(k)}+1.
\end{eqnarray*}
If we want (\ref{eq:I}) to be non zero, $s\ge0$ so that each term
$m^{(k)}$ must be at most $C_{\beta^{(k)}}+1$. This forces each
$m^{(k)}$ to be exactly $C_{\beta^{(k)}}+1$, indeed 
\begin{eqnarray*}
\sum_{i=1}^{r}(C_{\beta^{(i)}}+1) & = & \sum_{i=1}^{r}\left(c_{1}(X)\cdot\beta^{(i)}-1\right)\\
 & = & \left(\sum_{i=1}^{r}c_{1}(X)\cdot\beta^{(i)}\right)-r\\
 & = & C_{\beta}+2-r.\\
 & = & \sum_{i=1}^{r}m^{(i)}.
\end{eqnarray*}
Finally, the coefficient of $Y$ in $\ev_{*}\left(\frac{1}{1-\psi}[D]^{virt}\right)$
is
\[
\frac{1}{r!}\prod_{k=1}^{r}\left((C_{\beta^{(k)}}+1)I_{\beta^{(k)},(C_{\beta^{(k)}}+1)}(T)\right).
\]
\end{proof}
Let us go back to Equation (\ref{eq:Bigsum}). We are interested in
the coefficient of $Y$ in $J_{\beta,(C_{\beta})}$ and also in each
term 
\begin{equation}
\left(\prod_{j=i+1}^{C_{\beta}}(Y+j)\right)J_{\beta,(i)},\,\,\,i=0,...,C_{\beta}-1.\label{eq:Terms}
\end{equation}
For $J_{\beta,(C_{\beta})}$, by Lemma \ref{lem:The-multiple-of}
and Remark \ref{rem:H0andH2part} the required coefficient is
\[
\sum\frac{1}{2}\prod_{k=1}^{2}\left((C_{\beta^{(k)}}+1)I_{\beta^{(k)},(C_{\beta^{(k)}}+1)}(T)\right),
\]
where the sum runs over all the ordered partitions $(\beta^{(1)},\beta^{(2)})$
of $\beta$, with two non zero summands. The reason why we take ordered
partitions is the following. Since the marked points in $\M{1+r}(Y,0)$
are ordered, the two spaces
\[
\M 3(Y,0)\times_{Y}\Mbar_{(m^{(k)})}^{Y}(X,\beta^{(k)})\times_{Y}\Mbar_{(m^{(3-k)})}^{Y}(X,\beta^{(3-k)}),\,\,k\in\{1,2\},
\]
are isomorphic, but not the same if $\beta^{(1)}\neq\beta^{(2)}$.
So, we have to compute the contribution of each of them. For $(Y+C_{\beta})J_{\beta,(C_{\beta}-1)}$,
since $J_{\beta,(C_{\beta}-1)}$ has trivial $H^{0}$ coefficient
as noted in Remark \ref{rem:H0andH2part}, we can ignore the class
$Y$ in $(Y+C_{\beta})$. So the coefficient of $Y$ is
\[
C_{\beta}\sum\frac{1}{3!}\prod_{k=1}^{3}\left((C_{\beta^{(k)}}+1)I_{\beta^{(k)},(C_{\beta^{(k)}}+1)}(T)\right),
\]
where the sum runs over all the ordered partitions of $\beta$ with
three non zero summands. For any other term in (\ref{eq:Terms}) we
proceed in the same way. We get that the contribution of those terms
to (\ref{eq:Bigsum}) is the coefficient of $Y$ in $J_{\beta,(i)}$
times the number $\prod_{j=i+1}^{C_{\beta}}j=\frac{C_{\beta}!}{i!}$.
At the very end, we get that $I_{\beta,(C_{\beta}+1)}(T)$ is equal
to
\[
C_{\beta}!I_{1,\beta}^{X}(\mathrm{pt})-\sum_{K}\frac{C_{\beta}!}{(C_{\beta}+2-r_{K})!}\frac{1}{r_{K}!}\prod_{k=1}^{r_{K}}(C_{\beta^{(k)}}+1)I_{\beta^{(k)},(C_{\beta^{(k)}}+1)}(T),
\]
or, equivalently,
\begin{equation}
C_{\beta}!I_{1,\beta}^{X}(\mathrm{pt})-\sum_{K}\left(\begin{array}{c}
C_{\beta}\\
r_{K}-2
\end{array}\right)\frac{1}{r_{K}(r_{K}-1)}\prod_{k=1}^{r_{K}}(C_{\beta^{(k)}}+1)I_{\beta^{(k)},(C_{\beta^{(k)}}+1)}(T),\label{eq:abx}
\end{equation}
where the sum is taken among all the ordered partitions $K$ of $\beta$:
\[
K=(\beta^{(1)},...,\beta^{(r_{K})})\,\mathrm{such\,that}\,\sum_{k=1}^{r_{K}}\beta^{(k)}=\beta,\,\beta^{(k)}>0,r_{K}\ge2.
\]

\begin{rem}
Note that $I_{\beta,(C_{\beta}+1)}(T)$ does not depend on $Y$, but
only on $X$ and $\beta$. To prove that, we can use a simple induction
argument on the maximal length $max(\beta)$ of all the partitions
of $\beta$. If $max(\beta)=1$, i.e., $\beta$ is primitive, then
\[
I_{\beta,(C_{\beta}+1)}(T)=C_{\beta}!I_{1,\beta}^{X}(\mathrm{pt}).
\]
In the general case, $I_{\beta,(C_{\beta}+1)}(T)$ is a combination
of $I_{\beta^{(k)},(C_{\beta^{(k)}}+1)}(T)$ and other terms independent
of $Y$. But $I_{\beta^{(k)},(C_{\beta^{(k)}}+1)}(T)$ is independent
of $Y$ by induction, since clearly $max(\beta^{(k)})<max(\beta)$.
\end{rem}

\begin{example}
Let us give an example of a calculation using (\ref{eq:abx}). If
$X=\mathbb{P}^{3}$, then $I_{1,\beta}^{\mathbb{P}^{3}}(\mathrm{pt})=\frac{1}{(n!)^{4}}$
where $\beta=n[\mathrm{line}]$ by Example \ref{exa:Pand}. By a simple
calculation we see that $\OC(1,\mathbb{P}^{3})=2$. If $\beta$ is
the class of a conic, the unique partition is the sum of two lines,
so
\begin{eqnarray*}
\OC(2,\mathbb{P}^{3}) & = & \frac{C_{2}!}{2^{4}}-\frac{C_{2}!}{(C_{2}+2-2)!}\frac{1}{2}\prod_{k=1}^{2}(C_{1}+1)\OC(1,\mathbb{P}^{3})\\
 & = & 45-\frac{1}{2}\cdot3\cdot2\cdot3\cdot2\\
 & = & 27,
\end{eqnarray*}
as stated in Introduction.
\end{example}

\section{Applications}

\label{sec:Applications}In this section $X$ will be a homogeneous
variety, so by Proposition \ref{lem:Enumerative} $\OC(\beta,X)$
coincides with $I_{\beta,(C_{\beta}+1)}(T)$. To compute $\OC(\beta,X)$,
we need $I_{1,\beta}^{X}(\mathrm{pt})$. The opposite direction is
also possible: once we know $\OC(\beta,X)$ for some $\beta$, then
we can get $I_{1,\beta}^{X}(\mathrm{pt})$. For example, no point
of $\overline{M}_{0,1}(\mathbb{P}^{1},n)$ represents a birational
stable map if $n\ge2$, so by the proof of Proposition \ref{lem:Enumerative}
we expect $\OC(1,\mathbb{P}^{1})=1$ and $\OC(n,\mathbb{P}^{1})=0$
for $n\ge2$. Equation (\ref{eq:abx}) implies immediately $I_{1,1}^{\mathbb{P}^{1}}(\mathrm{pt})=1$,
whilst for $n\ge2$ the only non zero term of the sum
\[
\sum_{K}\frac{C_{\beta}!}{(C_{\beta}+2-r_{K})!}\frac{1}{r_{K}!}\prod_{i=1}^{r_{K}}(C_{\beta_{i}}+1)\OC(\beta_{i},\mathbb{P}^{1})
\]
appears when $K=(1,...,1)$. So the entire sum is equal to
\begin{eqnarray*}
\frac{C_{\beta}!}{(C_{\beta}+2-n)!}\frac{1}{n!}\prod_{i=1}^{n}(C_{1}+1)\OC(1,\mathbb{P}^{1}) & = & \frac{C_{\beta}!}{n!}\frac{1}{n!}1=\frac{C_{\beta}!}{(n!)^{2}}.
\end{eqnarray*}
Finally, the equation
\[
\OC(n,\mathbb{P}^{1})=C_{n}!I_{1,n}^{\mathbb{P}^{1}}(\mathrm{pt})-\frac{C_{n}!}{(n!)^{2}}
\]
implies $I_{1,n}^{\mathbb{P}^{1}}(\mathrm{pt})=\frac{1}{(n!)^{2}}$.

Using the same technique, we prove $I_{1,1}^{\mathbb{P}^{s}}(\mathrm{pt})=1$
for every $s\ge1$. Indeed, for $n=1$ (\ref{eq:abx}) reduces to
$\OC(1,\mathbb{P}^{s})=(s-1)!I_{1,1}^{\mathbb{P}^{s}}(\mathrm{pt})$.
So, it is enough to prove the following
\begin{prop}
$\OC(1,\mathbb{P}^{s})=(s-1)!$.
\end{prop}

\begin{proof}
Let $Y$ be a hypersurface of degree $d\gg0$. All spaces $D_{(m)}^{Y}(\mathbb{P}^{s},1)$
are empty. So by Fact \ref{fact:No_excess}, the number of osculating
lines is $c_{C_{\beta}+1}(\mathcal{P}^{C_{\beta}}(Y))\ev^{*}([Y]^{\vee})$,
where $C_{\beta}=s-1$. It a general fact that $\M 0(\mathbb{P}^{s},1)$
and $\M 1(\mathbb{P}^{s},1)$ are canonically isomorphic to, respectively,
the Grassmannian $G=G(2,s+1)$ of lines in $\mathbb{P}^{s}$ and its
universal family. There exists a rank 2 tautological vector bundle
$\mathcal{E}$ on $G$ such that $\M 1(\mathbb{P}^{s},1)\cong\mathbb{P}(\mathcal{E})$.
Moreover, $\psi$ coincides with the first Chern class of the relative
cotangent bundle of the natural map $\pi:\M 1(\mathbb{P}^{s},1)\ra G$.
From the exact sequence (\ref{eq:jetbundle}) we get, by a simple
recursion,
\begin{eqnarray*}
c_{s}(\mathcal{P}^{s-1}(Y)) & = & c_{1}(\mathbb{L}^{\otimes s-1}\otimes\ev^{*}(\mathcal{O}_{\mathbb{P}^{s}}(d)))\cdot c_{s-1}(\mathcal{P}^{s-2}(Y))\\
 & = & \prod_{i=0}^{s-1}c_{1}(\mathbb{L}^{\otimes i}\otimes\ev^{*}(\mathcal{O}_{\mathbb{P}^{s}}(d)))\\
 & = & d\xi\prod_{i=1}^{s-1}(i\psi+d\xi),
\end{eqnarray*}
where $\xi:=c_{1}(\ev^{*}(\mathcal{O}_{\mathbb{P}^{s}}(1)))$. By
definition $\ev^{*}([Y]^{\vee})=\frac{1}{d}\xi^{s-1}$, so
\begin{eqnarray*}
c_{s}(\mathcal{P}^{s-1}(Y))\ev^{*}([Y]^{\vee}) & = & \left(d\xi\prod_{i=1}^{s-1}(i\psi+d\xi)\right)\frac{1}{d}\xi^{s-1}\\
 & = & \xi^{s}\prod_{i=1}^{s-1}(i\psi+d\xi)
\end{eqnarray*}
Since $\xi^{i}=0$ if $i>s$, we have $c_{s}(\mathcal{P}^{s-1}(Y))[Y]^{\vee}=\xi^{s}(s-1)!\psi^{s-1}$.
Moreover, it is known that $\psi=\pi^{*}c_{1}(\mathcal{E}^{\vee})-2\xi$
(see, e.g., \cite[Theorem 11.4]{MR3617981}), so
\[
c_{s}(\mathcal{P}^{s-1}(Y))[Y]^{\vee}=(s-1)!\xi^{s}\pi^{*}c_{1}(\mathcal{E}^{\vee})^{s-1}.
\]
The degree of the zero cycle $\xi^{s}\pi^{*}c_{1}(\mathcal{E}^{\vee})^{s-1}$
is equal to the number of lines through a point and $s-1$ general
linear subspaces of codimension $2$. To prove that such number is
$1$, we can use Schubert calculus as explained in \cite[Chapter 4]{MR3617981}.
The Schubert cycle of lines through a codimension $2$ linear subspace
is $\sigma_{(1,0)}$. The Schubert cycles of lines through a point
is $\sigma_{(s-1,0)}$. Using Pieri's formula, for each integer $k\le s-1$
we have $(\sigma_{(1,0)})^{k}\cdot\sigma_{(s-1,0)}=\sigma_{(s-1,k)}$.
Finally 
\[
\xi^{s}\pi^{*}c_{1}(\mathcal{E}^{\vee})^{s-1}=(\sigma_{(1,0)})^{s-1}\cdot\sigma_{(s-1,0)}=\sigma_{(s-1,s-1)}=1.
\]
\end{proof}
On the other hand in the case $n=2$, (\ref{eq:abx}) reduces to
\begin{eqnarray}
\OC(2,\mathbb{P}^{s}) & = & C_{2}!I_{1,2}^{\mathbb{P}^{s}}(\mathrm{pt})-\frac{1}{2}(C_{1}+1)^{2}\OC(1,\mathbb{P}^{s})^{2}\label{eq:OC2}\\
 & = & (2s)!I_{1,2}^{\mathbb{P}^{s}}(\mathrm{pt})-\frac{1}{2}(s(s-1)!)^{2}.\nonumber 
\end{eqnarray}
We have seen in Example \ref{exa:osculatingOfPlaneCurve} that $\OC(2,\mathbb{P}^{2})=1$.
The case $\OC(2,\mathbb{P}^{3})=27$ was proved by Darboux. By (\ref{eq:OC2})
we have $I_{1,2}^{\mathbb{P}^{s}}(\mathrm{pt})=\frac{1}{2^{s+1}}$
for $s=2,3$.

\subsection{Computational aspects}

The aim of this subsection is to give a \textit{Wolfram Mathematica}
code to compute explicitly $\OC(\beta,X)$ for some $X$.

If $H^{2}(X,\Z)=\Z$, then by Poincarè duality all the effective homology
classes of a curve are multiple of a unique homology class. In those
cases, in the formula (\ref{eq:abx}) $K$ is equivalent to an ordered
partition of an integer $n$, where $n$ is a multiple of the cohomology
class generating $H^{2}(X,\Z)$. We constructed a code for the case
$X=\mathbb{P}^{s}$, in the following way. Using the command $\mathtt{IntegerPartitions[n,\{2,n\}]}$,
we get all \textbf{unordered} partitions of $n$. Let $K$ be one
of those partition, given by a \textit{list} of numbers. So we implement
Equation (\ref{eq:abx}) recursively, but we multiply by $\mathtt{Length[Permutations[K]]}$
to correct the fact that the partitions are unordered. The final code
is the following ($\OC(n,s)$ is $\OC(n[\mathrm{line}],\mathbb{P}^{s})$):

\noindent\begin{minipage}[t]{1\columnwidth}%
\centering\begin{verbatim}

OC[n_, s_]:=((s+1)*n-2)!/n!^(s+1)-Sum[(((s+1)*n-2)!/((s+1)*n
-Length[K])!)*(Length[Permutations[K]]/Length[K]!)*Product
[((s+1)*K[[i]]-1)*OC[K[[i]],s],{i,1,Length[K]}],
{K,IntegerPartitions[n,{2,n}]}];

\end{verbatim}%
\end{minipage} We have made several computer checks using \textsf{GROWI} \cite{GROWI}.
All results were as expected.

The code we give for $\OC(n[\mathrm{line}],\mathbb{P}^{s})$ can be
generalized for other varieties. What really changes is that we have
to find a way to write all the partitions of an effective homology
class $\beta$. Let $t>0$ be an integer and let $V\subset\Z^{\times t}$
be the convex cone generated by the coordinate vectors $(1,0,...,0),(0,1,0,...,0),...,(0,...,0,1)$.
For any $\beta=(\beta_{1},...,\beta_{t})\in V$, in Figure \ref{Figura}
we defined the function $\mathtt{SetP[\beta\_]}$ which gives all
unordered partitions of $\beta$ with summands in $V$. Such a function
is useful in the case that the variety $X$ has the effective cone
of $1$-cycles generated by the coordinate vectors, for example when
$X$ is homogeneous (see, e.g., \cite[Proposition 3.3,3.4]{Betti}).

Now we explain briefly the code in Figure \ref{Figura}. To avoid
confusion with notation, we will see our procedure just in an example,
leaving the general case as a formality. Suppose that $\beta$ is
a $t$-uple of non negative integers, for example $t=2$ and $\beta=(3,4)$.
Using the command $\mathtt{Catenate@@\{Array[Table[\#,\beta[[\#]]]\&,Length[\beta]]\}}$
we construct from $(3,4)$ the following \textit{list}:
\begin{equation}
\{\mathbf{1},\mathbf{1},\mathbf{1},\mathbf{2},\mathbf{2},\mathbf{2},\mathbf{2}\}.\label{eq:12}
\end{equation}
We applied \textbf{bold} to indicate that those are not numbers, but
symbols. So, we have the symbol $\mathbf{1}$ three times, and the
symbol $\mathbf{2}$ four times. Now we use the command $\mathtt{SetPartitions[]}$
that generates a list of all the partitions (as a list) of (\ref{eq:12}).
An example of such a list is
\begin{equation}
\{\{\mathbf{1}\},\{\mathbf{1},\mathbf{2},\mathbf{2},\mathbf{2}\},\{\mathbf{1},\mathbf{2}\}\}.\label{eq:1122}
\end{equation}
We use the commands $\mathtt{DeleteDuplicatesBy[...,Sort]}$ to remove
all the repetitions, and $\mathtt{Delete[...,1]}$ to avoid the first
partition, which is the trivial one given in (\ref{eq:12}). We denote
by $\mathtt{ParL[\beta\_]}$ the set of the partitions in the form
(\ref{eq:1122}). 

After that, we define the command $\mathtt{Conv[list\_,t\_]}$. It
converts a list of symbols like $\{\mathbf{1},\mathbf{2},\mathbf{2},\mathbf{2}\}$
into a $t$-uple of numbers. For example, in the list $\{\mathbf{1}\}$
the symbol $\mathbf{1}$ appears one time, and $\mathbf{2}$ appears
zero times. So, $\mathtt{Conv[\{\mathbf{1}\},2]}=(1,0)$. In the same
spirit we have $\mathtt{Conv[\{\mathbf{1},\mathbf{2},\mathbf{2},\mathbf{2}\},2]}=(1,3)$
and $\mathtt{Conv[\{\mathbf{1},\mathbf{2}\},2]}=(1,1)$. So, (\ref{eq:1122})
will be converted in the following partition of $\beta$:
\begin{equation}
\{(1,0),(1,3),(1,1)\}.\label{eq:(1,0)}
\end{equation}
Finally we define $\mathtt{SetP[\beta\_]}$, which is the list of
all partitions of $\beta$ in the form (\ref{eq:(1,0)}). We constructed
it just by applying $\mathtt{Conv[,]}$ to $\mathtt{ParL[\beta\_]}$.

Once we have $\mathtt{SetP[\beta\_]}$, we still need $C_{\beta}$
and $I_{1,\beta}^{X}(\mathrm{pt})$. When $X=\mathbb{P}^{s_{1}}\times...\times\mathbb{P}^{s_{t}}$,
in Figure \ref{Figura} they are represented, respectively, by the
commands $\mathtt{Cb[\beta\_,s\_]}$ and $\mathtt{Ix[\beta\_,s\_]}$.
We used that $I_{1,\beta}^{X_{1}\times\cdots\times X_{t}}(\mathrm{pt})=I_{1,\beta_{1}}^{X_{1}}(\mathrm{pt})\cdots I_{1,\beta_{t}}^{X_{t}}(\mathrm{pt})$,
see \cite[2.5]{kontsevich1994gromov}. At the very end, we have in
the last line a formula for $\OC(\beta,X)$. We used again $\mathtt{Length[Permutations[K]]}$,
where $K\in\mathtt{SetP[\beta]}$, to compute each partition with
the correct multiplicity. The command to get the value of $\OC((\beta_{1},...,\beta_{t}),\mathbb{P}^{s_{1}}\times...\times\mathbb{P}^{s_{t}})$,
where $s_{i},\beta_{i}\ge0$ for all $i$, is
\[
\mathtt{OC[\{\beta_{1},\ldots,\beta_{t}\},\{s_{1},\ldots,s_{t}\}].}
\]
For example, to compute $\OC((3,4),\mathbb{P}^{5}\times\mathbb{P}^{6})$
we use $\mathtt{OC[\{3,4\},\{5,6\}]}$. The result is precisely $1237651772190153893157497812054065\times10^{7}$.
A \textit{Wolfram Mathematica notebook} of this code can be provided
upon request.
\begin{figure}
\centering\includegraphics[scale=0.4]{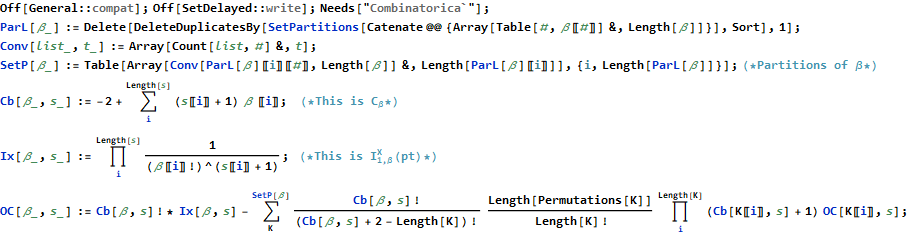}\caption{\label{Figura}Osculating curves of $X=\mathbb{P}^{s_{1}}\times...\times\mathbb{P}^{s_{t}}$.}
\end{figure}

\bibliographystyle{amsalpha}
\bibliography{1C__Users_magno_Dropbox_Universita_Ricerche_Asymptotic_Lines_Article_Osculating_conics_ref}

\end{document}